\documentclass[12pt]{amsart}
\usepackage{amsmath}
\usepackage{amsthm}
\usepackage{amssymb}
\usepackage{color}
\usepackage{indentfirst}
\usepackage{graphicx}
\usepackage{bbding}
\usepackage{hyperref}

\begin{document}

\newtheorem{theorem}{Theorem}
\newtheorem{lemma}{Lemma}
\newtheorem{corollary}{Corollary}
\newtheorem{proposition}{Proposition}
\newtheorem*{Ithm}{Ibragimov's Theorem}

\theoremstyle{definition}
\newtheorem{definition}{Definition}
\newtheorem{example}{Example}
\newtheorem{remark}{Remark}

\newenvironment{prof}[1][Proof]{\noindent\textit{#1}\quad }

\newcommand{\ind}{\mathop{\mathrm{index}}}
\parindent2em

\title{THREE STRONGLY HYPERBOLIC METRICS ON PTOLEMY SPACES}

\author{Yingqing Xiao}
\address{Yingqing Xiao, College of Mathematics and Econometrics, Hunan University, Changsha, 410082, China}
\email{ouxyq@hnu.edu.cn}

\author{Zhanqi Zhang}
\address{Zhanqi Zhang (Corresponding author), College of Mathematics and Econometrics, Hunan University, Changsha, 410082, P. R. China}
\email{rateriver@sina.com}
\maketitle

\textbf{Abstract:} Recently, strongly hyperbolic space as certain analytic enhancements of Gromov hyperbolic space was introduced by B. Nica and J. \v{S}pakula. In this note, we prove that the $\log$-metric $\log(1+d)$ on a Ptolemy space $(X,d)$ is a strongly hyperbolic metric. Using our result, we construct three metrics on a Ptolemy metric space and prove they are strongly hyperbolic.



\section{Introduction}
In the field of geometric function theory, the hyperbolic metric plays an important role. In higher dimensional Euclidean spaces, the hyperbolic metric exists only in balls and half-spaces and the lack of hyperbolic metric in general domains has been a primary motivation for introducing the so-called hyperbolic-type metrics  in the sense of Gromov. For example, $\widetilde{j}$-metric, Apollonian metric, Seittenranta's metric, half apollonian metric, scale-invariant Cassinian metric and M$\mathrm{\ddot{o}}$bius-invariant Cassinian metric (see [\cite{Beardon1,Hasto1,Hasto2,Hasto3,Ibragimov1, Ibragimov3, Ibragimov4,Seittenranta,Vuorinen}] and the references therein).  All these metrics are defined in terms of distance functions and can be classified into one point metrics or two-point metrics based on the number of boundary points used in their definitions. Recently, in the paper \cite{AZW}, the authors proposed an approach to construct a metric from the one-point metrics. More precisely, let $(X,d)$ be an arbitrary metric space. For each $p\in X$, they defined a distance function $\tau_p$ on $X \setminus\{p\}$, by
$$
\tau_p(x,y)=\log(1+2\frac{d(x,y)}{\sqrt{d(p,x)}\sqrt{d(p,y)}})
$$
and proved that for each $p\in X$, the distance function $\tau_p$  is Gromov hyperbolic with $\delta=\log3 +\log2$.  In fact, the following more general distance function was first introduced by O. Dovgoshey, P. Hariri and M. Vuorinen in \cite{Vuorinen2}.
$$
h_{D,c}(x,y)=\log(1+c\frac{d(x,y)}{\sqrt{d_D(x)d_D(y)}})
$$
where $D$ is a nonempty open set in a metric space $(X, d) $ and $d_D(x)=\mathrm{dist}(x, \partial D) , c\geq 2$. They shown that $h_{D,c}$ is a metric and $2$ is the best possible.

Although hyperbolicity yields a very satisfactory theory, for certain analytic purposes, hyperbolicity by itself is not enough, and one needs certain enhancements. In the paper \cite{NJ},  the authors introduced the notion of strongly hyperbolic space and given certain enhancements. They shown that strongly hyperbolic spaces are Gromov hyperbolic spaces that are metrically well-behaved at infinity, and under weak geodesic assumptions, the strongly space are strongly bolic as well. They shown that $\mathrm{CAT}(-1)$ spaces are strongly hyperbolic and also shown that the Green metric defined by a random walk on a hyperbolic group is strongly hyperbolic. Since the strongly hyperbolic space has better properties, it is interesting  to determine which hyperbolic metric in geometric function theory is a strongly hyperbolic metric or to construct a strongly hyperbolic metric on a given metric space. We consider this problem in Ptolemy spaces in this paper.

Firstly, we show that the  $\log$-metric of a Ptolemy space is a strongly hyperbolic metric. That is, we show that if $(X,d)$ is a Ptolemy space, then $(X,\log(1+d))$ is a strongly hyperbolic space. Using our result, we can show that the metric space $(X, S_p)$ is also a strongly hyperbolic space. Here $$S_p(x,y)=\log(1+
 \frac{d(x,y)}{[1+d(x,p)][1+d(y,p)]})$$
 for a fix point $p\in X$ and $x,y\in X$.

Secondly, motivated by the recent works of A. G. Aksov, Z. Ibragimov and W. Whiting in \cite{AZW},  we construct a strongly hyperbolic metric on a Ptolemy metric space.  To formulate the results of our paper, for each $p\in X$, we define a distance function $\chi_p$ on $X \setminus\{p\}$, by
$$
\chi_p(x,y)=\log(1+\frac{d(x,y)}{d(p,x)d(p,y)}).
$$
 We prove that if $(X,d)$ is a Ptolemy space,  for each $p\in X$, the distance function $\chi_p$  is a strongly hyperbolic metric.
We also consider the distortion of the above metric $\chi_p$ under M\"{o}bius maps of a punctured ball in $\mathbb{R}^n$.

\section{Strongly hyperbolic metrics on Ptolemy spaces}
We begin by recalling some basic notions and facts. Let $X$ be a metric space, fix a base point $o\in X$, the Gromov product of $x, x'\in X$ with respect to $o$ is defined as
$$
(x|x')_o:=\frac{1}{2}(|ox|+|ox'|-|xx'|).
$$
Note that $(x|x')_o\geq 0$ by the triangle inequality.
\begin{definition}[Gromov]\label{def-1}
A metric spaces $X$ is \emph{$\delta$-hyperbolic}, where $\delta \geq 0$, if
$$
(x|y)_o\geq \min\{(x|z)_o,(z|y)_o\}-\delta
$$
for all $x, y, z,o\in X$.
\end{definition}
In the paper \cite{NJ}, the authors given the following enhancements of hyperbolicity.
\begin{definition}[\cite{NJ}, Definition 4.1]We say that a metric space is \emph{strongly hyperbolic} with parameter $\epsilon>0$ if
$$
\exp(-\epsilon(x|y)_o)\leq \exp(-\epsilon(x|z)_o)+\exp(-\epsilon(z|y)_o)
$$
for all $x, y, z, o\in X$; equivalently, the four-point condition
$$
\exp(\frac{\epsilon}{2}(|xy|+|zt|))\leq \exp(\frac{\epsilon}{2}(|xz|+|yt|))+\exp(\frac{\epsilon}{2}(|xt|+|zy|))
$$
holds for all $x, y, z, t\in X$.
\end{definition}
The authors mentioned the motivation for considering this notion of strongly hyperbolic is the following theorem in the paper \cite{NJ},
\begin{theorem}[\cite{NJ}, Theorem 4.2] \label{NJThereom}
Let $X$ be a strongly hyperbolic space with parameter $\epsilon$. Then X is an $\epsilon$-good, $\log2/\epsilon$-hyperbolic space. Furthermore, $X$ is strongly bolic provided that $X$ is roughly geodesic.
\end{theorem}
Strongly bolic metric spaces was considered by V. Lafforgue in
\cite{Lafforgue} in relation with conjecture of Baum-Connes. Here for hyperbolic spaces $(X,d)$ which are roughly geodesic, strong bolicity in the sense of Lafforgue \cite{Lafforgue} amounts to the following:
for every $\eta,r>0$,  there exists $R>0$  such that $d(x,y)+d(z,t)\leq r$ and $d(x,z)+d(y,t)\geq R$   imply that
 $d(x,t)+d(y,z)\leq d(x,z)+d(y,t)+\eta$.

From the above theorem \ref{NJThereom}, we know that the strongly hyperbolic space has better properties than general hyperbolic spaces.
Thus it is interesting to construct a strongly hyperbolic metric on a metric space.
%
%

\begin{definition} A metric space $(X,d)$ is called \emph{Ptolemy space} if the following Ptolemy inequality
$$
d(x_1,x_2)d(x_3,x_4)\leq d(x_1,x_4)d(x_2,x_3)+d(x_1,x_3)d(x_2,x_4)
$$
holds for all quadruples $x_1,x_2,x_3,x_4\in X$.
\end{definition}

\begin{lemma}\label{lemma-2.1}
Suppose $(X,d)$ is a metric space and $x_i\in X$ for $i=1,2,3,4$. Then
$$
d(x_1,x_2)+d(x_3,x_4)\leq d(x_1,x_3)+d(x_1,x_4)+d(x_2,x_3)+d(x_2,x_4).
$$
\end{lemma}
\begin{prof}
By the triangle inequality, we have
\begin{align}\label{ine-2.1.2}   \nonumber
d(x_1,x_2)\leq d(x_1,x_3)+d(x_3,x_2),\\  \nonumber
d(x_1,x_2)\leq d(x_1,x_4)+d(x_4,x_2),\\   \nonumber
d(x_3,x_4)\leq d(x_3,x_1)+d(x_1,x_4),\\   \nonumber
d(x_3,x_4)\leq d(x_3,x_2)+d(x_2,x_4).      \nonumber
\end{align}
We sum the above four inequalities and obtain that
$$
d(x_1,x_2)+d(x_3,x_4)\leq d(x_1,x_3)+d(x_1,x_4)+d(x_2,x_3)+d(x_2,x_4).
$$\qed
\end{prof}

\begin{theorem}\label{keytheorem}
Suppose that $(X,d)$ is a Ptolemy space, then the metric space $(X,\log(1+d))$ is a strongly hyperbolic space with parameter $2$.
\end{theorem}
\begin{prof}
Let $x_1,x_2,x_3,x_4\in X$, we introduce the following notations for convenience.
$ \rho_{ij}=\log(1+d(x_i,x_j))$, $d_{ij}=d(x_i,x_j)$ for all $i,j\in \{1,2,3,4\}$. Thus
$$
\rho_{ij}=\log(1+d_{ij}).
$$

Now, we need to show that
$$
e^{(\rho_{12}+\rho_{34})}\leq e^{(\rho_{13}+\rho_{24})}+e^{(\rho_{14}+\rho_{23})},
$$
which is equivalent to the following inequality,
$$
(1+d_{12})(1+d_{34})\leq (1+d_{13})(1+d_{24})+(1+d_{14})(1+d_{23}).
$$
Notice that $(X,d)$ is a Ptolemy space, by Lemma \ref{lemma-2.1}, we have
\begin{equation*}
\begin{split}
(1+d_{12})(1+d_{34})&\;=1+d_{12}+d_{34}+d_{12}d_{34}\\
&\;\leq 2+d_{13}+d_{24}+d_{14}+d_{23}+d_{14}d_{23}+d_{13}d_{24}\\
&\;=(1+d_{13})(1+d_{24})+(1+d_{14})(1+d_{23}).\\
\end{split}
\end{equation*}
Thus, we show that the metric space $(X, \log(1+d))$ is a strongly hyperbolic hyperbolic space with parameter $\epsilon=2$.
\qed
\end{prof}

Let $(X,d)$ be any metric space, fix a base point $p\in X$, and the following distance function $s_p$ was considered in the paper \cite{BHX},
 $$
 s_p(x,y)= \frac{d(x,y)}{[1+d(x,p)][1+d(y,p)]}
 $$
for $x,y\in X$.  Sometimes this is a distance function, but in general it may not satisfy the triangle inequality. In this paper, we have the following result.
\begin{theorem}Suppose $(X,d)$ is a Ptolemy space and $p\in X$. Then $(X, s_p)$ is also a Ptolemy space.
\end{theorem}
\begin{proof}
Firstly, we prove that $s_p$ is a metric.  Obviously, $s_p(x,y)\geq 0$, $s_p(x,y)=s_p(y,x)$ and $s_p(x, y)=0$ if and only
if $x=y$. So it is enough to show that the triangle inequality holds. That is, for all $x, y, z\in X\setminus\{p\}$,
$$
s_p(x, y)\leq s_p(x, z)+s_p(z, y),
$$
which is equivalent to
$$
d(x,y)[1+d(z,p)]\leq d(x,z)[1+d(y,p)]+d(y,z)[1+d(x,p)].
$$
Since $(X,d)$ is a Ptolemy space, the above inequality holds naturally, which implies that $s_p$ is a metric on $X$.

Now, we show that $(X, s_p)$ also is a Ptolemy space.
For any $x_i\in X$ for $i=1,2,3,4$. Set $p_i=1+d(p,x_i)$ and $d_{ij}=d(x_i,x_j)$,
thus $s_p(x_x,x_j)=d_{ij}/p_ip_j$
for $i,j\in \{1,2,3,4\}$.
Since $(X,d)$ is a Ptolemy space, we have
$$
d_{12}d_{34}\leq d_{13}d_{24}+d_{14}d_{23},
$$
Thus
$$
\frac{d_{12}d_{34}}{p_1p_2p_3p_4}\leq \frac{d_{13}d_{24}}{p_1p_2p_3p_4}+\frac{d_{14}d_{23}}{p_1p_2p_3p_4}.
$$
That is
$$
s_p(x_1,x_2)s_p(x_3,x_4)\leq s_p(x_1,x_3)s_p(x_2,x_4)+s_p(x_1,x_4)s_p(x_2,x_3),
$$
which implies that $(X, s_p)$ also is a Ptolemy space.
\end{proof}

Using $s_p$, we define the following metric  $S_p$ on $X$
by
$$
S_p(x,y)=\log(1+s_p(x,y)).
$$
According to Theorem \ref{keytheorem}, we have the following result.
\begin{theorem}
Supppose $(X,d)$ is a Ptolemy and $p\in X$. The metric space $(X, S_p)$ is a strongly hyperbolic space with parameter $\epsilon=2$. Thus $(X, S_p)$  is a $\log2/2$-hyperbolic space.
\end{theorem}

Suppose $(X,d)$ is a metric space. For each $p\in X$, A. G. Aksov, Z. Ibragimov and W. Whiting  defined a distance function $\tau_p$ on $X\setminus\{p\}$
in \cite{AZW} by
$$
\tau_p(x,y)=\log(1+2\frac{d(x,y)}{\sqrt{d(p,x)}\sqrt{d(p,y)}}).
$$
They obtained the following result.
\begin{theorem}[\cite{AZW}, Theorem 2.1 and Lemma 4.1]Let $(X, d)$ be a Ptolemy space and let $p\in X$ be an
arbitrary point. Then the distance function $\tau_p$ is a metric on $X\setminus\{p\}$. In
particular, the space $(X\setminus\{p\}, \tau_p)$ is Gromov hyperbolic with $\delta=\log 3+\log2$.
\end{theorem}

Motivated by the definition of $\tau_p$, for each $p\in X$, we define a distance function $\chi_p$ on $X\setminus\{p\}$
by
$$
\chi_p(x,y)=\log(1+\frac{d(x,y)}{d(p,x)d(p,y)}).
$$
Usually, $\chi_p$ is not a metric on $X\setminus\{p\}$. But, when $(X,d)$ is a Ptolemy space, we have the following result.
\begin{theorem}Let $(X, d)$ be a Ptolemy metric space and let $p\in X$ be an
arbitrary point. Then the distance function $\chi_p$ is a metric on $X\setminus\{p\}$.
\end{theorem}
\begin{proof}
Obviously, $\chi_p(x,y)\geq 0$, $\chi_p(x,y)=\chi_p(y,x)$ and $\chi_p(x, y)=0$ if and only
if $x=y$. So it is enough to show that the triangle inequality holds. That is, for all $x, y, z\in X\setminus\{p\}$,
$$
\chi_p(x, y)\leq\chi_p(x, z)+\chi_p(z, y),
$$
which is equivalent to
$$
\frac{d(x,y)}{d(x,p)d(y,p)}\leq \frac{d(x,z)}{d(x,p)d(z,p)}
+\frac{d(y,z)}{d(y,p)d(z,p)}
+\frac{d(x,z)d(y,z)}{d(z,p)^2d(x,p)d(y,p)}.
$$
That is
\begin{align}\label{eq0-metric}
d(x,y)d(z,p)\leq d(x,z)d(y,p)+d(y,z)d(x,p)+\frac{d(x,z)d(y,z)}{d(z,p)}.
\end{align}
Since $(X,d)$ is a Ptolemy space, the above inequality (\ref{eq0-metric}) holds naturally, which completes the proof.
\end{proof}

\begin{lemma}\label{lemma-2.2}
Suppose $(X,d)$ is a Ptolemy metric space and $x_i\in X$ for $i=0,1,2,3,4$.
Set $p_i=d(x_0,x_i)$ and $d_{ij}=d(x_i,x_j)$ for $i,j\in \{1,2,3,4\}$.
Then
$$
p_3p_4d_{12}+p_1p_2d_{34}\leq p_1p_3d_{24}+p_2p_4d_{13}+p_2p_3d_{14}+p_1p_4d_{23}.
$$
\end{lemma}
\begin{proof}
By the Ptolemy inequality, we have
\begin{align}\label{ine-2.1.2}  \nonumber
p_3p_4d_{12}\leq p_3p_1d_{24}+p_3p_2d_{14},\\ \nonumber
p_3p_4d_{12}\leq p_4p_2d_{13}+p_1p_4d_{23},\\ \nonumber
p_1p_2d_{34}\leq p_1p_3d_{24}+p_1p_4d_{23},\\ \nonumber
p_1p_2d_{34}\leq p_2p_4d_{13}+p_2p_3d_{14}. \nonumber
\end{align}
We sum the above four inequalities and obtain that
$$
p_3p_4d_{12}+p_1p_2d_{34}\leq
p_1p_3d_{24}+p_2p_4d_{13}+p_2p_3d_{14}+p_1p_4d_{23}.
$$
\end{proof}

Using the above lemma \ref{lemma-2.2}, we obtain the following result.
\begin{theorem}Let $(X, d)$ be a Ptolemy metric space and let $p\in X$ be an
arbitrary point. Then the metric space $(X\setminus\{p\}, \chi_p)$ is strongly hyperbolic space with parameter $2$.
Thus $(X\setminus\{p\}, \chi_p)$  is $\log2/2$-hyperbolic space.
\end{theorem}
\begin{proof}Let $x_1,x_2,x_3,x_4\in X\setminus\{p\}$, we introduce the following notations for convenience.
$d_{ij}=d(x_i,x_j)$, $p_i=d(p,x_i)$ and $\rho_{ij}=\chi_p(x_i,x_j)$ for $i,j\in \{1,2,3,4\}$. Thus
$$
\rho_{ij}=\log(1+\frac{d_{ij}}{p_ip_j})
$$
for $i,j\in \{1,2,3,4\}$.
Now, we need to show that
$$
e^{(\rho_{12}+\rho_{34})}\leq e^{(\rho_{13}+\rho_{24})}+e^{(\rho_{14}+\rho_{23})},
$$
which is equivalent to the following inequality
\begin{align}\nonumber
(1+\frac{d_{12}}{p_1p_2})(1+\frac{d_{34}}{p_3p_4})\leq (1+\frac{d_{13}}{p_1p_3})(1+\frac{d_{24}}{p_2p_4})\\     \nonumber
+(1+\frac{d_{14}}{p_1p_4})(1+\frac{d_{23}}{p_2p_3}). \nonumber
\end{align}
That is
\begin{align}
\frac{d_{12}}{p_1p_2}+\frac{d_{34}}{p_3p_4}+\frac{d_{12}}{p_1p_2}\frac{d_{34}}{p_3p_4}
&\leq\frac{d_{13}}{p_1p_3}+\frac{d_{24}}{p_2p_4}+\frac{d_{13}}{p_1p_3}\frac{d_{24}}{p_2p_4}\\ \nonumber
&+\frac{d_{14}}{p_1p_4}+\frac{d_{23}}{p_2p_3}+\frac{d_{14}}{p_1p_4}\frac{d_{23}}{p_2p_3}+1,\nonumber
\end{align}
which is equivalent to the following inequality
\begin{align}\label{strongly}
p_3p_4d_{12}+p_1p_2d_{34}+d_{12}d_{34}&\leq p_2p_4d_{13}+p_1p_3d_{24}+d_{13}d_{24}\\ \nonumber
&+p_2p_3d_{14}+p_1p_4d_{23}+d_{14}d_{23}\\ \nonumber
&+p_1p_2p_3p_4.\nonumber
\end{align}
Since $(X,d)$ is a Ptolemy space, we have
$$
d_{12}d_{34}\leq d_{13}d_{24}+d_{14}d_{23}.
$$
From Lemma \ref{lemma-2.2}, we have
$$
p_3p_4d_{12}+p_1p_2d_{34}\leq p_2p_4d_{13}+p_1p_3d_{24}+p_2p_3d_{14}+p_1p_4d_{23}.
$$
Thus, the above inequality \ref{strongly} holds, which implies that
$(X\setminus\{p\}, \chi_p)$ is a strongly space with parameter $2$.  From Theorem \ref{NJThereom}, we know that $(X\setminus\{p\}, \chi_p)$  is $\log2/2$-hyperbolic space. \end{proof}

\section{Distortion property under M\"{o}bius transformations}
In the following, we use the notation $\mathbb{R}^n,n\geq 2$ for the Euclidean-dimensional space. 
The Euclidean distance between $x, y\in \mathbb{R}^n$ is denoted by $|x-y|$. Given  $x\in\mathbb{R}^n$ and $r>0$, the open ball centered at $x$ with radius $r$ is denoted by $B^n(x,r):=\{y\in \mathbb{R}^n: |x-y|< r\}$.  Denote by $\mathbb{B}^n := B^n(0,1)$, the unit ball in $\mathbb{R}^n$. One of our objectives in this section is to study the distortion property of our metric under M\"{o}bius maps from a punctured ball onto another punctured ball. Distortion properties of the scale-invariant Cassinian metric of the unit ball under M$\mathrm{\ddot{o}}$bius maps has been studied in \cite{Ibragimov3}. Recently, in the \cite{Msahoo},   M. R. Mohapatra and S. K. Sahoo also considered the distortion of the $\widetilde{\tau}$-metric under Mobius maps of a punctured ball.

%
\begin{theorem}Let $a \in \mathbb{B}^n$ and $f: \mathbb{B}^n \setminus \{0\}\rightarrow \mathbb{B}^n \setminus\{a\} $ be a M$\ddot{o}$bius map with $f(0)=a$. Then for $x,y\in\mathbb{B}^n \setminus \{0\}$,  we have
$$
\chi_0(x,y)\leq\chi_a(f(x),f(y))\leq\chi_0(x,y)-\log(1-|a|^2).
$$
The equalities hold if and only if $a=0$.
\end{theorem}
\begin{proof}
If $a=0$, the proof is trivial since $f(x)=Ax$ for some orthogonal matrix $A$. Now we assume that $a\neq 0$. Let $\sigma$ be the inversion in the sphere
$\mathbb{S}^{n-1}(a^{*},r)=\{x\in \mathbb{R}^n: |x-a^{*}|=r\}$,  where
$$
a^{*}=\frac{a}{|a|^2},r=\sqrt{|a^{*}|^2-1}=\frac{\sqrt{1-|a|^2}}{|a|}.
$$
Note that the sphere $\mathbb{S}^{n-1}(a^{*},r)$ is orthogonal to $\mathbb{S}^{n-1}$ and that $\sigma(a)=0$. In particular, $\sigma$ is a M$\mathrm{\ddot{o}}$bius map with $\sigma(\mathbb{B}^n \setminus \{a\})= \mathbb{B}^n \setminus \{0\}$. Recall that
$$
\sigma(x)=a^{*}+\big(\frac{r}{|x-a^{*}|}\big)^2(x-a^{*}).
$$
Then $\sigma\circ f$ is an orthogonal matrix (see, for example, [\cite{Beardon}, Theorem 3.5.1(i)]). In particular,
$$
|\sigma(f(x))-\sigma(f(y))|=|x-y|.
$$
By computation, we have
$$
|\sigma(x)-\sigma(y)|=\frac{r^2|x-y|}{|x-a^{*}||y-a^{*}|}.
$$
Thus
$$
|\sigma(f(x))-\sigma(f(y))|=\frac{r^2|f(x)-f(y)|}{|f(x)-a^{*}||f(y)-a^{*}|}=|x-y|,
$$
which implies that
$$
|f(x)-f(y)|=\frac{|x-y|}{r^2}|f(x)-a^{*}||f(y)-a^{*}|.
$$

Since $f(0)=a$, we have
$$
|f(x)-a|=\frac{|f(x)-a^{*}||a-a^{*}|}{|a^{*}|^2-1}|x|\quad\text{and}\quad|f(y)-a|=\frac{|f(y)-a^{*}||a-a^{*}|}{|a^{*}|^2-1}|y|.
$$
Notice that
$$
\chi_0(x,y)=\log(1+\frac{|x-y|}{|x||y|})
$$
and
$$
\chi_a(f(x),f(y))=\log(1+\frac{|f(x)-f(y)|}{|f(x)-a||f(y)-a|}).
$$
We have
\begin{align}\nonumber
\chi_a(f(x),f(y))&=\log(1+\frac{|f(x)-f(y)|}{|f(x)-a||f(y)-a|})\\ \nonumber
&=\log(1+\frac{|x-y|}{|x||y|}\frac{|a^{*}|^2-1}{|a-a^{*}|^2})\\ \nonumber
&=\log(1+\frac{1}{1-|a|^2}\frac{|x-y|}{|x||y|}).\nonumber
\end{align}
Since $|a|<1$, we have $1\leq\frac{1}{1-|a|^2}$.  Thus
$$
1+\frac{|x-y|}{|x||y|}\leq1+\frac{1}{1-|a|^2}\frac{|x-y|}{|x||y|}\leq\frac{1}{1-|a|^2}+\frac{1}{1-|a|^2}\frac{|x-y|}{|x||y|}.
$$
So
$$
\chi_0(x,y)\leq\chi_a(f(x),f(y))\leq\chi_0(x,y)-\log(1-|a|^2).
$$
Obviously, the equalities hold if and only if $a=0$.
\end{proof}

\noindent\textbf{Acknowledgements.} This work was supported by the National Natural Science Foundation of China under grant Nos.\,11301165,11571099.

\end{document}